\documentclass[10pt,a4paper, xcolor]{article}
\usepackage{geometry,amsthm,graphics,amssymb,amsmath,latexsym}
\usepackage[all,2cell]{xy} \UseAllTwocells \SilentMatrices
\usepackage{amsfonts}
\usepackage{amssymb,amsthm, amsmath,graphicx,bm,ebezier,amsthm}
\usepackage{color,enumerate}
\usepackage{url}
\usepackage{hyperref,mathtools}
\usepackage{authblk}
\usepackage{caption}
\usepackage{listings}

\newtheorem{theorem}{Theorem}
\newtheorem{definition}[theorem]{Definition}
\newtheorem{proposition}[theorem]{Proposition}
\newtheorem{corollary}[theorem]{Corollary}
\newtheorem{lemma}[theorem]{Lemma}

\theoremstyle{remark}

\newtheorem{example}[theorem]{Example}
\newtheorem{algorithm}[theorem]{Algorithm}
\newtheorem{remark}[theorem]{Remark}

\def\w{{\omega}}
\def\gcd{{\mathrm gcd}}

\def\N{\mathbb{N}}
\def\R{\mathbb{R}}
\def\Z{\mathbb{Z}}
\def\Q{\mathbb{Q}}
\def\d{\mathrm{d}}

\def\int{\mathrm{int}}

\newcommand{\D}{{\mathfrak D}}

\providecommand{\keywords}[1]{\smallskip\noindent{\small\emph{Keywords:} #1}}
\providecommand{\msc}[1]{\noindent{\small\emph{MSC-class:} #1}}

\title{On divisor-closed submonoids and  minimal distances in
finitely generated monoids}

\author[1]{J. I. Garc\'{\i}a-Garc\'{\i}a\thanks{Corresponding author. E-mail address: ignacio.garcia@uca.es.
Supported by
the project MTM2014-55367-P, which is funded by Ministerio de Economía y Competitividad and Fondo Europeo de Desarrollo Regional FEDER,
and by
FQM-366 (Junta de Andaluc\'{\i}a).}}

\author[$ $]{D. Mar\'{\i}n-Arag\'on\thanks{
Email address: daniel.marinaragon@alum.uca.es.
}}

\author[3]{M. A. Moreno-Fr\'{\i}as
\thanks{E-mail address: mariangeles.moreno@uca.es.
Supported by
the project MTM2014-55367-P, which is funded by Ministerio de Economía y Competitividad and Fondo Europeo de Desarrollo Regional FEDER,
and by
 FQM-298 (Junta de Andaluc\'{\i}a).}}
\affil[1,3]{Departamento de Matem\'aticas, Universidad de C\'adiz}

\date{}

\begin{document}

\maketitle

\begin{abstract}
\noindent
We study the lattice of divisor-closed submonoids of finitely generated cancellative
commutative monoids. In case the monoid is an affine semigroup,
we give  a geometrical characterization of such submonoids in terms of its cone.
Finally, we use our results to give an algorithm for computing $\Delta^*(H)$, the set
of minimal distances of $H$.


\msc{
13A05, 
11R27, 
20M14, 
20M13, 
52B11. 
}

\end{abstract}

\section{Introduction}

All semigroups, monoids and groups appearing in this paper are commutative and thus in the sequel we omit this adjective when we refer to one of these
objects. We denote by $\Z$, $\N$, and $\Q$ the set of integers,  nonnegative integers, and rationals, 
respectively.

Non-unique Factorization Theory was initiated in 1990s, it has its
origins in algebraic number theory (see \cite{Narkiewicz}), and its most important goal is to classify
non-uniqueness of factorization occurring for different commutative monoids and domains.
Since then,
the study of 
factorizations in integral
domains and monoids has become an active area of interest, and, 
in particularly, a class that  plays a central role in factorization theory are Krull monoids (see \cite{grillet}).
In \cite{Geroldinger-Halter}, 
we can find a large amount of results concerning properties of factorizations, and a detailed study of Krull monoids. In that work, 
parameters and properties are defined and studied mainly from a theoretical perspective.
One paper that applies algorithmic methods, such as the found in \cite{Rosales3}, for the study of finitely generated commutative monoids to compute parameters of Non-unique Factorization Theory is \cite{atomicos};  
it is remarkable that the algorithms given are illustrated with examples of monoids where these parameters are computed.
Not only \cite{atomicos} is in this line of study, 
for instance, in \cite{elasticidad}, it is given a algorithm to compute the elasticity of a monoid,
the catenary degree is computed in \cite{oneill}, the tame degree in \cite{chapman1}, 
the $\omega$-primality in \cite{chapman2} and \cite{w-primalidad}, and the
$\Delta$-set  of a monoid in \cite{delta-set}.

This work 
follows the above-mentioned line, and it 
is motivated by the study of the 
set of minimal distances of a Krull monoid $H$ with finite class group $G$. 
This parameter, denoted by $\Delta^*(H)$, was first investigated in \cite{Geroldinger}, and,
among other works, it is also 
studied in \cite{Chapman-Chang}, \cite{Chapman-Schmid-Smith},
\cite{Gao-Geroldinger}, \cite{Geroldinger-Hamidoune}, and \cite{Geroldinger-Qinghai}.
A result of \cite{Chapman-Chang} allows us to determine the minimal
distance of certain Krull monoids with cyclic class group.
In \cite{Chapman-Schmid-Smith}, the set of minimal distances of $H$
is studied for Krull monoids with infinite class group.
It is well-known (see \cite{Geroldinger-Halter}) that
if a Krull monoid $H$ has class group $G$ with $G$ finite, then there is a constant
$M\in \N$ such that all the sets of lengths are almost arithmetical multiprogressions
with bound $M$ and difference $d\in  \Delta^*(H)$.
Recently, in \cite{Geroldinger-Qinghai},
it is proved  that
$\max\Delta^*(H)\leq \{{\rm exp}(G)-2,{\rm r}(G)-1\}$ with ${\rm exp}(G)$ and ${\rm r}(G)$ as defined in 
\cite[Appendix A]{Geroldinger-Halter} and the equality holds if every class of
$G$ contains a prime divisor.
Usually, this parameter is usually defined for Krull monoids, but
in \cite{Kainrath-Lettl}  the monoid of
non-negative solutions of systems of linear Diophantine equations
was characterized
as being the same as
Krull monoids with finitely generated (torsion free) quotient group. Also, using the chapter by
Chapman and Geroldinger in \cite{Anderson}, one can show that the sets of factorization
lengths in a Krull domain with finite divisor class group are identical to those
in a block monoid over a finite Abelian group, which is by definition a finitely
generated reduced cancellative monoid (and thus strongly reduced).
This fact, for example,  is used in \cite{elasticidad} to calculate the elasticity in
full affine semigroups and the elasticity in Diophantine monoids.
Besides, in \cite[Corollary 16]{atomicos}, it is proved that every
finitely generated cancellative monoid
is atomic, and
for this reason it is possible to translate the concept of divisor-closed submonoids to
finitely generated cancellative monoid and study the computation of the set of
minimal distances in  these monoids.

The aim of this work is to provide a method for computing the
set of minimal distances
of a finitely generated cancellative monoid
by studying the structure of its set of divisor-closed submonoids.
Using the Archimedean components of the monoid, we show how to obtain these submonoids,
and, in Theorem \ref{componentes}, we prove that the set of divisor-closed submonoids has a structure
of complete finite lattice.
For finitely generated submonoids of
$\N^n$ (also known as affine semigroups), we describe geometrically their
set of divisor-closed submonoids,
we use this description to obtain a
system of generators of every divisor-closed submonoid, and we give a method
for computing their  sets of minimal distances. 
The method to obtain the divisor-closed submonoids of an affine semigroup
uses the polyhedral cone generated by the
affine semigroup  and requires the use of tools of
linear integer programming and topology.
The main result of this work is Theorem \ref{carash}, and, in this result,
we prove that every divisor-closed submonoid of
an affine semigroup
$H$ is equal to the intersection of a face of the cone generated by $H$ with $H$ and that
every divisor-closed submonoid of $H$ is of this form.
In others words, it proves that for every affine semigroup $H$ its lattice of  Archimedean components, its
lattice of divisor-closed submonoids and the lattice of faces of the polyhedral cone generated
by $H$ are isomorphic.
Another interesting result is Corollary \ref{checkdc} which characterizes the divisor-closed submonoids of a cancellative monoid $\tilde H$ in term of the divisor-closed submonoids of a associated affine monoid $H$.
So, we have two different methods to calculate
the set of divisor-closed submonoids of a finitely generated cancellative monoid and
thus for computing its set of minimal distances. These methods have been illustrated in this work with several examples, and the software used have been \cite{normaliz} and \cite{ecuacionesZpython}.

We follow our introduction with a short summary which outlines some basic definitions concerning to monoids.
In Section 3, the definition of divisor-closed monoid is presented and we express it as union of its Archimedean components. We turn our attention in Section 4 to affine semigroups, and we use geometrical tools for giving the lattice of divisor-closed submonoids. In Section 5, we compute the set of divisor-closed submonoids of a reduced monoid using the results of Section 4. Finally, in Section 6, we provide an algorithm for computing the set of minimal distance of finitely generated cancellative monoids.

\section{Notation and Definitions}
Factorization-theoretic notions are usually defined for multiplicative monoids, but we use additive notation for our aim.

The following definitions and results are in \cite{Rosales3}.
A  semigroup is a pair $(H,+)$, with $H$ a non-empty set and $+$ a binary operation defined on $H$ verifying the associative law. 
In addition, if there exists an element, which is usually denoted by $0$, in $H$ such that $a+0=0+a$ for all $a\in H$, we say that $(H,+)$ is a monoid.
Given a subset $A$ of a monoid $H$, the monoid generated by $A$, denoted by $\langle A\rangle$, is the least (with respect to inclusion) submonoid of $H$ containing $A$. When $H=\langle A\rangle $, we say that $H$ is generated by $A$ or that $A$ is a system of generators of $H$. The monoid $H$ is finitely generated if it has a finite system of generators. 
Recall that finitely generated submonoids of $\N^n$ are known as affine semigroups.

Given a  monoid $H$ we define the following binary relation over it:
\[b\leq_H a \textrm{ if } a=b+c \textrm{ for some }c\in H.\]
Clearly $\leq_H$ is reflexive and transitive, and if $b\leq_H a$, then
$b+c\leq_H a+c$ for all $c\in H$. 
Dealing with multiplicative written monoids, this notion is the same as the notion of divisivility, so
we say that $b$ divides $a$ if and only if $b\leq_H a$.

Define the binary relation $\mathcal N$ on $H$ by
\[ a \mathcal{N} b \textrm{ if there exist } k \textrm{ and } l\in\N\setminus\{0\} \textrm{ such that } ka\geq_H b \textrm{ and } lb\geq_H a.\]
It is easy to check that $\mathcal N$ is a congruence (a equivalence binary relation compatible with the addition) and that
if $H$ is finitely generated then $H/\mathcal N$ is finite. The elements of
$H/\mathcal N$ are known as the Archimedean components of $H$, and they are subsemigroups
of $H$. Every finitely generated monoid can be viewed as a finite lattice of its Archimedean components, and 
an ordering in $H/\mathcal N$ can be defined as follows: $S_1\leq S_2$ if and only if
for every $a\in S_1$ and every $b\in S_2$ the element $a+b$ belongs to $S_2$.
A monoid $H$ is Archimedean whenever it has only two Archimedean components, and, in such case, their components are $\{0\}$ and $H\setminus\{0\}$.

A  monoid $H$ is cancellative if whenever the equality $a+c=b+c$ holds for some
$a,b,c\in H$, then $a=b$.
In the sequel, we assume that all monoids appearing are  cancellative.
If $H$ is a finitely generated monoid, then
$H$ is isomorphic to $\N^p/\sim_M$ for some positive integer $p$ and some subgroup
$M$ of $\Z^p$, where $\sim_M$ is the congruence defined by $a\sim_M b$ if and only if $a-b\in M$.
The subgroups of $\Z^p$ are always determined by a set of defining equations, that is, if  $M$ is a subgroup of $\Z^p$ (denoted by $M\leq \Z^p$), there exists a set of equations of the form
\[
\begin{array}{lcl}
a_{11}x_1+\dots+a_{1p}x_p & \equiv & 0 \mod d_1,\\
 & \vdots & \\
a_{r1}x_1+\dots+a_{rp}x_p & \equiv & 0 \mod d_r,\\
a_{(r+1)1}x_1+\dots+a_{(r+1)p}x_p & = & 0,\\
 & \vdots & \\
a_{(r+k)1}x_1+\dots+a_{(r+k)p}x_p & = & 0\\
\end{array}
\]
such that an element is in $M$ if and only if it verifies the above equations.
By \cite[Proposition 3.1]{Rosales3}, if $H\cong \N^p/\sim_M$, then $H$ is
isomorphic to the submonoid of $\Z_{d_1}\times\dots\times\Z_{d_r}\times \Z^k$ generated by
\[
\begin{multlined}
\{  ([a_{11}]_{d_1},\dots,[a_{r1}]_{d_r},a_{(r+1)1},\dots,a_{(r+k)1}),\dots,\\
([a_{1p}]_{d_1},\dots,[a_{rp}]_{d_r},a_{(r+1)p},\dots,a_{(r+k)p})\}
\end{multlined}
\]
where $[a]_d$ denotes the equivalence class of $a$ in the finite Abelian group $\Z_d$.
A monoid is called reduced whenever it does not have nontrivial units. In \cite{Rosales3}, it proved that $\N^p/\sim_M$ is reduced if and only if $M\cap \N^p=\{0\}$, and this occurs if there exist $b_1,\dots,b_p\in\N\setminus \{0\}$ such that every element $x$ of $M$ satisfies the homogeneous equation $b_1x_1+\dots +b_p x_p=0$.

Given two monoids $H_1$ and $H_2$ a map $f$ from $H_1$ to $H_2$ is a monoid morphism
if $f(0)=0$ and $f(a+b)=f(a)+f(b)$ for all $a,b\in H_1$.

\section{Divisor-closed submonoids}

A submonoid $S$ of $H$ is called a
divisor-closed submonoid of $H$ if $a\in S$, $b\in H$, and $b$ divides $a$ imply that
$b\in S$ (see \cite{Geroldinger-Qinghai}).
Trivially, the submonoids of $H$, $\{0\}$ and $H$, are
divisor-closed submonoids of $H$.
With the additive notation, this notion translates 
as follows:
a submonoid $S$ of $H$ is called divisor-closed if $b+c\in S$ implies  $b,c\in S$. 
One of our goal is to compute the set of divisor-closed submonoids of a given  finitely generated cancellative  monoid $H$; this set of submonoids is denoted by $\D (H)$.

\begin{example} Given $H=\langle  5,7 \rangle$, the submonoid $S=\langle
7 \rangle$ is not divisor-closed because $35=5+5+5+5+5+5+5\in S$, but
$5$ divides $35$ and $5 \notin  S$.
\end{example}

\begin{example}
Consider $S$ the submonoid of $\N^2$ generated by $(1,0)$. This submonoid is trivially
a divisor-closed submonoid of $\N^2$. In contrast, $S'=\langle (5,0)\rangle$ is not
divisor-closed, since $(5,0)=(2,0)+(3,0)$ and $(2,0)\not\in S'$ but $(5,0)\in S'$.
\end{example}

We prove now that if $H$ is finitely generated, then all every divisor-closed submonoid of $H$ is  generated by a subset of the system of generators of $H$.

\begin{proposition}\label{dcfg}
Let $H$ be a finitely generated  monoid with $G=\{g_1,\dots,g_p\}$
one of its system of generators.
Then, every divisor-closed submonoid of $H$ is finitely generated and
has a system of generators contained in $G$.
\end{proposition}
\begin{proof}
Let $S$ be a divisor-closed submonoid of $H$.
Every element $s\in S$ can be expressed as $\sum_{i=1}^p \lambda^s_i g_i$ for some
$\lambda^s_i\in \N$. Since $S$ is divisor-closed, we obtain that if for some
$s\in S$ the coefficient $\lambda^s_i$ is not null, then $\lambda^s_i g_i\in S$, and
therefore $g_i\in S$. So, we have that
\[G_S=\{g_i\mid \textrm{ there exists }s\in S \textrm{ such that }\lambda^s_i\neq 0 \}
\subset S, \]
and thus, $\langle G_S\rangle \subseteq S$.
Trivially
$S\subseteq \langle G_S\rangle$, and therefore $S=\langle G_S\rangle $.
\end{proof}

From the above result we obtain that every finitely generated monoid $H$ has only a finite number of divisor-closed submonoids. For instance, 
if $H=\langle g_1,g_2,g_3\rangle$ is a monoid, the lattice of candidates
to be divisor-closed submonoids of $H$ are represented in this figure:
\[\xymatrix{
 &  H  & & \\
\langle g_1, g_2 \rangle \ar[ur]&    & & \langle g_2,g_3 \rangle\ar[llu]\\
  & \langle g_1,g_3 \rangle \ar[uu]& \langle g_2 \rangle \ar[llu]\ar[ur]& \\
\langle g_1 \rangle \ar[uu] \ar[ur]&    & & \langle g_3 \rangle \ar[uu]\ar[llu]\\
 & &\{0\} \ar[llu]\ar[uu] \ar[ur]& \\
}\]
But, not every submonoid generated by a subset of the
system of generators of $H$ is divisor-closed.
Take
$H=\langle (3,0),(0,3),(2,2)\rangle \subseteq \N^2$ and $S=\langle (2,2) \rangle$; 
since the element $(6,6)$ is in $S$, but
$2(3,0)+2(0,3)=(6,6)$ and  $2(0,3)\notin S$,
 $S$ is not divisor-closed.
Thus, we need a method
to check
whether such a submonoid $S$ is or is not divisor-closed.

Adding the zero element to an  Archimedean component of a monoid $H$ we obtain a submonoid of $H$.
These submonoids are not necessarily finitely generated; therefore, in general,
they are not divisor-closed.
For example, consider $H=\N^2$. 
Its Archimedean components are: the subsemigroups
$C=\{(x_1,x_2)| x_1\geq 1 \textrm{ and } x_2\geq 1\}$, 
$C_1=\{(x,0)|x\geq 1\}$, $C_2=\{(0,x)|x\geq 1\}$, and $\{(0,0)\}$.
Adding the zero element to $C$, $C_1$, and $C_2$, we obtain the submonoids: $C\cup\{(0,0)\}$, $C_1\cup\{(0,0)\}$ and $C_2\cup\{(0,0)\}$.
The submonoids $C_1\cup\{(0,0)\}$ and $C_2\cup\{(0,0)\}$ are finitely generated, but
$C\cup\{(0,0)\}$ is not finitely generated. Therefore, $C\cup\{(0,0)\}$ is not a divisor-closed submonoid of $\N^2$.

We see now 
the relation between the Archimedean components of a
finitely generated monoid and its divisor-closed submonoids.

\begin{theorem}\label{componentes}
Every divisor-closed submonoid $S$ of a finitely generated monoid $H$ can be expressed as a
union of Archimedean components of $H$.
Furthermore, there exists an Archimedean component $\hat S$ such that
$S=\cup \{S'| S' \textrm{ is an Archimedean component of }H\textrm{ and }S'\leq \hat S\}$.
\end{theorem}
\begin{proof}
Let $s\in S$, we prove now that if $S'$ is the Archimedean component of $H$ that
contains $s$, then $S'\subset S$. If $s'\in S'$, since $s\mathcal N s'$, there exist
$k,l\in\N\setminus\{0\}$ such that $k s'\geq_H s$ and $ l s\geq_H s'$. From
$ l s\geq_H s'$, we obtain that there exists $h\in H$ such that $l s=s'+h$.
Since $s\in S$, the element $l s$ is also in $S$, and thus $l s=s'+h\in S$. Using
now that $S$ is divisor-closed, we deduce that $s'\in S$ and therefore
$S'\subseteq S$.
So, every Archimedean component verifying that $S\cap S'\neq \emptyset$ verifies that
$S'\subset S$.
Since $H$ is the union of its Archimedean components,
$S = \cup\{S'\in H/\mathcal N\mid S'\cap S\neq \emptyset \}$.

Consider $\hat S$  the supreme  of the elements of
$\mathcal S=\{S'\in H/\mathcal N\mid S'\cap S\neq \emptyset \}$
in the lattice $H/\mathcal N$.
Since $H/\mathcal N$ is finite, $\mathcal S$ is also finite.
Consider for every element in $\mathcal S$  one of its elements,
make the sum of all of them  and call
the result $\hat s$.
Since the Archimedean components that intersect with $S$ are contained in $S$, the element
$\hat s$ belongs to $S$. Furthermore, using that $\hat S$ is the supreme
of the elements of $\mathcal S$,
the element $\hat s$ belongs to $\hat S$. Thus, $\hat S\subset S$.
We also have that for every Archimedean component $S'$ such that $S'\leq \hat S$, if
$s'\in S'$, then $s'+\hat s\in \hat S\subset S$. Hence, $s'\in S$ and therefore $S'\subset S$.
So, an Archimedean component $S'$ intersects with $S$ if and only if $S'\leq \hat S$, and therefore
$S=\cup \{S'| S' \textrm{ is an Archimedean component of }H\textrm{ and }S'\leq \hat S\}$.
\end{proof}

A consequence of Theorem \ref{componentes} is that the set of
divisor-closed submonoids of $H$ can be represented as a finite lattice with respect to the inclusion ordering.

\begin{example}
We consider again the monoid $H=\langle (3,0),(0,3),(2,2)\rangle \subseteq \N^2$.
Its Archimedean components
are
$\{(0,0)\}$, $\{(3x,0)|x>1\}$, $\{(0,3x)|x>1\}$ and $H$. Thus,  the set of divisor-closed
submonoids of $H$ is 
$\D(H)=\{\{(0,0)\},\{(3x,0)|x>1\}\cup\{(0,0)\}, \{(0,3x)|x>1\}\cup\{(0,0)\},H\}$.
\end{example}

From  Theorem \ref{componentes}, we also obtain  that
the only divisor-closed submonoids of an Archimedean monoid $H$ are $\{0\}$ and $H$.
In particular,
numerical semigroups (submonoids $S$ of $\N$ such that $\N\setminus S$ is finite) are always Archimedean, and, therefore, they do not have any non-trivial divisor-closed submonoids.

There exists an algorithm to compute the Archimedean components of a finitely
generated monoid (see \cite[\S 13]{Rosales3}) from a given presentation (see also \cite[\S 5]{Rosales3} for more information on presentations
of monoids).
Using Theorem \ref{componentes}, the computation of the divisor-closed submonoids of a finitely generated monoid can be obtained algorithmically.

\section{Divisor-closed submonoids of affine semigroups}\label{sec4}

In this section, we focus our attention on  affine semigroups for giving a different
method
to obtain the set of divisor-closed submonoids of a given 
affine monoid.
The main advantage of our approach is that
we provide a geometrical
description of the set of divisor-closed submonoids in terms of the faces of
the polyhedral cone generated by the monoid.

Denote by $\Q_+$ the set of non-negative rational numbers and
$e_i$ the $i$th element of the canonical basis of $\Q^n$. The dot product of two elements $x, y\in \Q^n$ is written as $x\cdot y$ and is equal to $\sum_{i=1}^n x_i y_i$.

We say that $\mathbf C$ is a polyhedral cone if there exist $d_1,\dots,d_t\in\Q^n$ such that
$\mathbf C=\{\sum_{i=1}^t \lambda_i d_i| \lambda_1,\dots,\lambda_t \in\Q_+\}$
or, equivalently, if there exist
$f_1,\dots,f_t\in\Q^n$ such that
$\mathbf C=\{x\in\Q^n | f_1^T\cdot x\leq 0,\dots,f_t^T\cdot x\leq 0\}$
(see \cite[\S 7.2]{schrijver}).
In this case, the set $\{d_1,\dots,d_t\}$ is known as a system of generators of the polyhedral
cone $\mathbf C$.
Every
vector $\w\in\R^n$ determines a face  of a polyhedral cone $\mathbf{C}$ as follows:
\[ {\rm face}_{\w}(\mathbf{C})= \{ x\in\mathbf{C}| x.\w\geq y.\w
\textrm{ for all } y\in\mathbf{C} \}.\]
In other words,  ${\rm face}_{\w}(\mathbf{C})$ is the set of points for which
the function $f_\w(x)=x.\w$ is maximum, we denote this maximum by $\delta_\w$
(see \cite[\S 8.3]{schrijver}).

For
every $\w\in\R^n$, since the origin belongs to $\mathbf{C}$,
we have $\delta_w\geq 0$.
Besides, if $x.\w>0$ for  $x\in \mathbf C$, then $2x\in\mathbf{C}$,
and therefore $(2x).\w=2(x.\w)$. This implies that
the function $f_\w(x)=x.\w$ has not maximum and therefore
${\rm face}_\w(\mathbf{C})=\emptyset$.
Hence, for every non-empty face $F={\rm face}_\w(\mathbf{C})$ of a polyhedral cone $\mathbf{C}$,
$\delta_\w=0$ and the origin belongs to $F$.
In particular, we have that every polyhedral cone has exactly one
vertex, the origin.

Every face $F={\rm face}_\w(\mathbf{C})$ of a polyhedral cone $\mathbf C$ is
determined by the inequalities of $\mathbf C$ and the inequalities $x.\w\leq 0$ and $-x.\w\leq 0$.
Thus, every  face $F$  of a polyhedral cone $\mathbf C$ is again polyhedral cone.
Furthermore, if $\{ d_1,\dots, d_r\}$ are the elements of
$\{ d_1,\dots, d_t\}$ fulfilling that $d_i\cdot w=0$, the face $F$ is equal to
$\{\sum_{i=1}^r \lambda_i d_i| \lambda_1,\dots,\lambda_r \in\Q_+\}$.
So, $\mathbf C$ has at most $2^t$ faces.
From \cite[page 10]{Sturmfels}, we have that if $F'$ is a face of $F$, then
$F'$ is also a face of $\mathbf C$ (transitivity of the relation \emph{is a face of}).

We denote the vectorial subspace spanned by a polyhedral cone $\mathbf C$ by $V_{\mathbf C}$.
The dimension of $V_{\mathbf C}$
is known as the  dimension of $\mathbf C$. Similarly,
we define the  dimension of a face of a polyhedral cone; we refer to the $d$-dimensional
faces of $\mathbf C$ as its $d$-faces. In a $n$-dimensional polyhedral cone,
its $(n-1)$-faces are known as its facets and its $1$-faces as its
extremal rays; the origin is the only $0$-face.
From \cite[\S 8.8]{schrijver} we have that every polyhedral cone
$\mathbf C$ is generated by its extremal rays.
Denote by $\mathfrak F (\mathbf C)$ the set of faces of $\mathbf C$.
It follows from \cite[\S 8.6]{schrijver} and \cite[page 30]{brondsted} that the partially ordered set
$(\mathfrak F (\mathbf C),\subset)$  is a complete finite lattice with the operations
\[
\begin{array}{c}
\inf (\mathfrak A) = \cap \{F \in \mathfrak F (\mathbf C)|F\in\mathfrak A\},\\
\sup (\mathfrak A)=\cap \{G \in \mathfrak F (\mathbf C)|
\forall  F\in\mathfrak A \textrm{ : } F\subset G\}
\end{array}
\]
for every $\mathfrak A\subset \mathfrak F(\mathbf C)$.

\begin{definition}
Let $H$ be an affine semigroup of $\N^n$. Define the rational cone of $H$ as
${\rm L}_{\Q_+}(H)=\{\sum_{i=1}^r\lambda_ih_i\mid r\in\N,h_i\in H,\lambda_i\in\Q_+\}$.
The set ${\rm L}_{\Q_+}(H)\cap \N^n$ is denote by $\mathcal C_H$.
\end{definition}

The set $\mathcal C_H$ is a submonoid of $\Q_+^n$ and $H\subseteq \mathcal C_H$.
If $\{g_1,\dots,g_p\}\subset \N^n$ is a system of generators of $H$, then
${\rm L}_{\Q_+}(H)=\{\sum_{i=1}^p\lambda_ig_i\in\Q^n\mid \lambda_i\in\Q_+\}$.
Since
the elements $g_1,\dots,g_p$ are rational, by \cite[Corollary 7.1a]{schrijver}
there exists a matrix $A\in\mathcal M_{m \times n}(\Q)$ with $m\in \N$ such that
${\rm L}_{\Q_+}(H)=\{x\in\Q^n\mid Ax \leq 0 \}$. So, the set $\mathcal C_H$ is equal to
$\{x\in\N^n\mid Ax \leq 0 \}$. By \cite[Theorem 16.4]{schrijver}, it is a
submonoid of $\N^n$ finitely generated by
${\rm Minimals_\leq}\{x\in\N^n\mid Ax \leq 0 \}$ with
$\leq$ the product ordering on $\N^n$. Hence, $\mathcal C_H$ is an affine semigroup of $\N^n$.

We proof now the first relation between the faces of a ${\rm L}_{\Q_+}(H)$ and the divisor-closed submonoids of $H$.

\begin{lemma}\label{lemaS}
Let $H\subset \N^n$ be an affine semigroup and
$F$ be a face of ${\rm L}_{\Q_+}(H)$. The semigroup $S=F\cap H$ is an affine divisor-closed submonoid of $H$.
\end{lemma}
\begin{proof}
Since  $F$ is a face of ${\rm L}_{\Q_+}(H)$, there exists $\w\in\R^n$ such that
$F={\rm face}_\w({\rm L}_{\Q_+}(H))=\{x\in {\rm L}_{\Q_+}(H) | x.\w=0 \}$ and
${\rm L}_{\Q_+}(H)\setminus F=\{x\in {\rm L}_{\Q_+}(H)| x.\w<0\}$.
Assume that there exist $a,b\in H$ such that $a+b\in S$.
Since $H\subset {\rm L}_{\Q_+}(H)$, $a.\w\leq 0$ and $b.\w\leq 0$. Furthermore,
since $S\subset F$, $(a+b).\w=0$. Using that $(a+b).\w=a.\w+b.\w$, we obtain that
$a.\w=b.\w=0$ which implies that $a,b\in F$ and therefore $a,b\in S$.
Furthermore if $G$ is a system of generators of $H$, a system of generators of $S$ is the set $G\cap F$. Thus, $S$ is a finitely generated and therefore it is an affine divisor-closed submonoid of $H$.
\end{proof}

\begin{example}
Let  $H=\langle (1,0),(1,2),(1,3),(1,7) \rangle$.
We check that $S=\langle(1,7)\rangle$ is a divisor-closed submonoid of $H$.
Assume $(\lambda,7\lambda)=a+b$, $a,b\in H$ and $\lambda \in \N$.
Then, $(\lambda,7\lambda)=
(a_1+b_1,0)+(a_2+b_2,2a_2+2b_2)+(a_3+b_3,3a_3+3b_3)+(a_4+b_4,7a_4+7b_4)$ with $a_i,b_i\in\N$. Therefore,
$$\left\{\begin{array}{rcl}
\lambda & = & a_1+b_1+a_2+b_2+a_3+b_3+a_4+b_4,\\
7\lambda & = & 2a_2+2b_2+3a_3+3b_3+7a_4+7b_4.
\end{array}\right.$$
Hence, $7(a_1+b_1+a_2+b_2+a_3+b_3+a_4+b_4)=2a_2+2b_2+3a_3+3b_3+7a_4+7b_4$ and
$a_i,b_i$ are zero for $i\leq 3$ and therefore $S$ is a divisor-closed submonoid of $H$.

It is not hard to prove that ${\rm L}_{\Q_+}(H)$ has four faces: $\{(0,0)\}$, $H$,
the $1$-face generated by $(1,0)$ and the $1$-face generated by $(1,7)$.
Clearly $S=\{ \lambda (1,7)|\lambda\in\Q_+\}\cap H$ and by Lemma \ref{lemaS} the
submonoid $S=\langle(1,7)\rangle$ is a divisor-closed submonoid of $H$.

{\bf decir quienes son las demás caras}
\end{example}

The following is an immediate consequence of the above lemma.

\begin{corollary}\label{corCono}
Let $\mathbf C$ be a polyhedral cone. Then, for every face $F$ of $\mathbf C$
the semigroup $F\cap \N^n$ is a divisor-closed submonoid of $\mathbf C\cap \N^n$.
\end{corollary}

We use the definitions and notation of \cite[Chapter 1.3]{brondsted} for relative interior
(${\rm ri}(\mathbf C)$) and
relative boundary (${\rm rb}(\mathbf C)$).
Let $\mathbf C\subset \Q^n$ be a polyhedral cone.
By the relative interior of $\mathbf C$ we mean the interior of $\mathbf C$ in $V_{\mathbf C}$ with the Euclidean topology;
denote it by ${\rm ri}(\mathbf C)$. Similarly, we define the relative boundary
of $\mathbf C$; denote it by ${\rm rb}(\mathbf C)$.
Since every polyhedral cone is a closed subset of $\R^n$, the
relative closure of $\mathbf C$ in $V_{\mathbf C}$
is simply the closure of $\mathbf C$ in $\Q^n$ with the Euclidean topology.

In order to prove Propostion \ref{carascono}, we prove the following three lemmas.

\begin{lemma}\label{boundary}
Let $\mathbf C\subset \Q_+^n$ be a polyhedral cone and consider
$\R^n$ with the Euclidean topology.
Then,  the relative boundary ${\rm rb} (\mathbf C)$ of $\mathbf C$  is equal to the
union of its facets.
\end{lemma}
\begin{proof}
See \cite[Theorem 4.3]{brondsted}.
\end{proof}

\begin{lemma}\label{interior}
Let $\mathbf C\subset \Q_+^n$ be a polyhedral cone and let $S$ be a submonoid of
$\mathbf C\cap \N^n$.
If $S$ is not contained in any proper face of $\mathbf C$ then
$S\cap  {\rm ri}(\mathbf C)\neq \emptyset$.
\end{lemma}
\begin{proof}
Since $\mathbf C$ is a polyhedral cone,
there exist $d_1,\dots,d_t\in \Q^n_+$ such that
$\mathbf C=\{\sum_{i=1}^t\lambda_i d_i\in\Q^n\mid \lambda_i\in\Q_+\}$.
For every face $F={\rm face}_\omega(\mathbf C)$ define $R_F=\{i\in\{1,\dots,t\}|\omega\cdot d_i=0\}$.
We have that if $\sum_{i=1}^t\lambda_i d_i\in F$ with $\lambda_i\in\Q_+$, then $\lambda_i=0$ for every $i\not \in R_F$.
Define 
\[R_S=\left\{ j\in\{1,\dots,t\}|\textrm{ there exists } s\in S\textrm{ such that }
s=\sum_{i=1}^t\lambda_i d_i\textrm{ with }\lambda_i\in\Q_+\textrm{ and }\lambda_j> 0
\right\}.\]
For every $j\in R_S$ we take $s_j\in S$ such that
$s_j=\sum_{i=1}^t\lambda_i^{j} d_i$
with $\lambda_i^{j}\in\Q_+$ and $\lambda_j^{j}> 0$.
Since $S$ is a monoid,  $s=\sum_{j\in R_S}s_j$ belongs to $S$.
The element $s$ is equal to
$\sum_{j\in R_S}(\sum_{i=1}^t \lambda_i^{j} d_i)=
\sum_{i=1}^t(\sum_{j\in R_S} \lambda_i^{j}) d_i= \sum_{i=1}^t \mu_id_i$ with $\mu_i=(\sum_{j\in R_S} \lambda_i^{j})>0$ for every $i\in R_S$.
Since $S$ is not contained in any proper face of $\mathbf C$, we have that
$R_S\not\subseteq R_F$ for every face $F$ of $\mathbf C$.
Thus, $s$ does not belong to any
proper face of $\mathbf C$.
So, $s$ is an element of $\mathbf C\cap S$ that it is not in any facet of
$\mathbf C$. By Lemma \ref{boundary}, we obtain that $s\in {\rm ri}(\mathbf C)$.
\end{proof}

Given $a,b\in\Q^n_+$, denote by $\measuredangle ab$ the angle between them, that is,
the value $0\leq {\rm arccos} (\frac {a\cdot b}{\|a\|.\|b\|}) \leq \pi/2$.

\begin{lemma}\label{angulo}
Let $\mathbf C$ be a polyhedral cone  contained in $\Q^n_+$ and
$a\in {\rm ri}(\mathbf C)$.
Then, for every $b\in \mathbf C\setminus\{0\}$ we have that
$0\leq \measuredangle ab < \pi/2$.
\end{lemma}
\begin{proof}
Assume that $a=(a_1,\dots,a_n)\in{\rm ri}(\mathbf C)$ and
let $b=(b_1,\dots,b_n)$ be an element of $\mathbf C\setminus\{0\}$.
Since $\mathbf C\subset \Q_+^n$, the element
$a$ has open
neighbourhoods contained in $\Q^n_+$.
Since $a,b\in\Q^n_+$, its dot product verifies $a\cdot b\geq 0$ and therefore
$0\leq \measuredangle ab \leq \pi/2$.
The angle is $\pi/2$ if and only if $a\cdot b=0$. This occurs if and only if
the sets $\{i\in\{1,\dots,n\}\mid a_i\neq 0\}$ and
$\{i\in\{1,\dots,n\}\mid b_i\neq 0\}$ are disjoint.
Since $\mathbf C$ is convex, the segment $\overline {ab}$ is contained in
$\mathbf C$ and therefore the straight line containing $a,b$ is a subset of 
$V_{\mathbf C}$.
Using that
$\{i\in\{1,\dots,n\}\mid a_i\neq 0\}\cap
\{i\in\{1,\dots,n\}\mid b_i\neq 0\}$ is empty, every neighbourhood of $a$ in
$V_{\mathbf C}$ contains points in $\Q^n\setminus\Q^n_+$, that is,
points with at least a negative coordinate ($a+\alpha(a-b)$ with $\alpha>0$).
This contradicts the fact that $a\in {\rm ri}(\mathbf C)$.
\end{proof}

We now prove  the reciprocal of Corollary \ref{corCono}.

\begin{proposition}\label{carascono}
Let $\mathbf C$ be a polyhedral cone contained in $\Q_+^n$.
If $S$ is a divisor-closed submonoid of $\mathbf C\cap \N^n$ then
there exists a face $F$ of $\mathbf C$ such that $S=F\cap \N^n$.
\end{proposition}
\begin{proof}
By \cite[Corollary 7.1a]{schrijver}, there exist $d_1,\dots,d_t\in \Q^n_+$ such that
$\mathbf C=\{\sum_{i=1}^p\lambda_i d_i\in\Q^n\mid \lambda_i\in\Q_+\}$.
Let $F$ be the smallest face of $\mathbf C$ containing $S$ (the intersection of  the faces containing $S$). Clearly, $S\subset F$.
The set $F$ is again a cone and $S$ is not contained in any proper face of $F$.
By Lemma \ref{interior}, there
exists $a\in S\cap {\rm ri} (F)$. For every $b\in F\cap \N^n$ we have $a\cdot b\geq 0$.
By Lemma \ref{angulo},
this implies that the angle $\gamma$ between these elements verifies that
$0\leq \gamma< \pi/2$.
Consider now the straight line $r$ containing $a$ and $b$. The segment
$\overline{ab}$ is contained in $F$. Using that $a\in{\rm ri}(F)$ there exists
$c\in\Q^n_+\cap {\rm ri}(F)$ such that $\overline{a b}\subsetneq \overline{ac}$
(see Figure \ref{fig1})
and verifying that $0 < \measuredangle bc<\pi/2$.
Define $\tau_a$ and $\tau_c$ the rays containing $a$ and $c$, respectively, 
and $\Pi$ the vectorial plane containing $a$
and $b$. This plane $\Pi$ also contains $c$, so there exists $r_c$ a straight line
contained in $\Pi$ parallel to the ray $\tau_c$ such that $b\in r_c$. Since
$0 < \measuredangle bc<\pi/2$, the lines $r_c$ and $\tau_a$ intersect in a point.
We call this point $p$. Using that $a,b,c\in \Q^n_+$ we obtain that $p$ is also
in $\Q^n_+$, and thus $p\in F$.
For $p$ there exist $\mu_1,\mu_2\in \Q_+$ such that
$\mu_1 a=b+\mu_2c=p\in\Q^n_+$.
Taking $k\in \N$ equal to the least common multiple of the denominators of
$\mu_1$, $\mu_2$ and the coordinates of $p$, we obtain that
$k\mu_1,k\mu_2\in \N$ and $k \mu_1 a=k (b+\mu_2c)=k p\in\N^n$.
Since $a\in S$ and $k\mu_1\in\N$, $k \mu_1 a$ is also in $S$.
Since $k b\in\N^n$, $k\mu_2 c\in\Q^n_+$ and $kb+k\mu_2 c\in\N^n$,
the element $k \mu_2 c$ is also in $\N^n$.
Using that $S$ is a divisor-closed submonoid of $\mathbf C\cap \N^n$,
$k \mu_1 a\in S$,  $kb,k\mu_2c\in F\cap \N^n$ and
$k\mu_1 a=k b+k \mu_2c$, we deduce  $k b,k\mu_2 c\in S$,
and therefore $b\in S$. Hence $F\cap \N^p\subset S$ and thus $F=S$.
\end{proof}

\begin{figure}[h]
\begin{center}
\includegraphics[scale=.5]{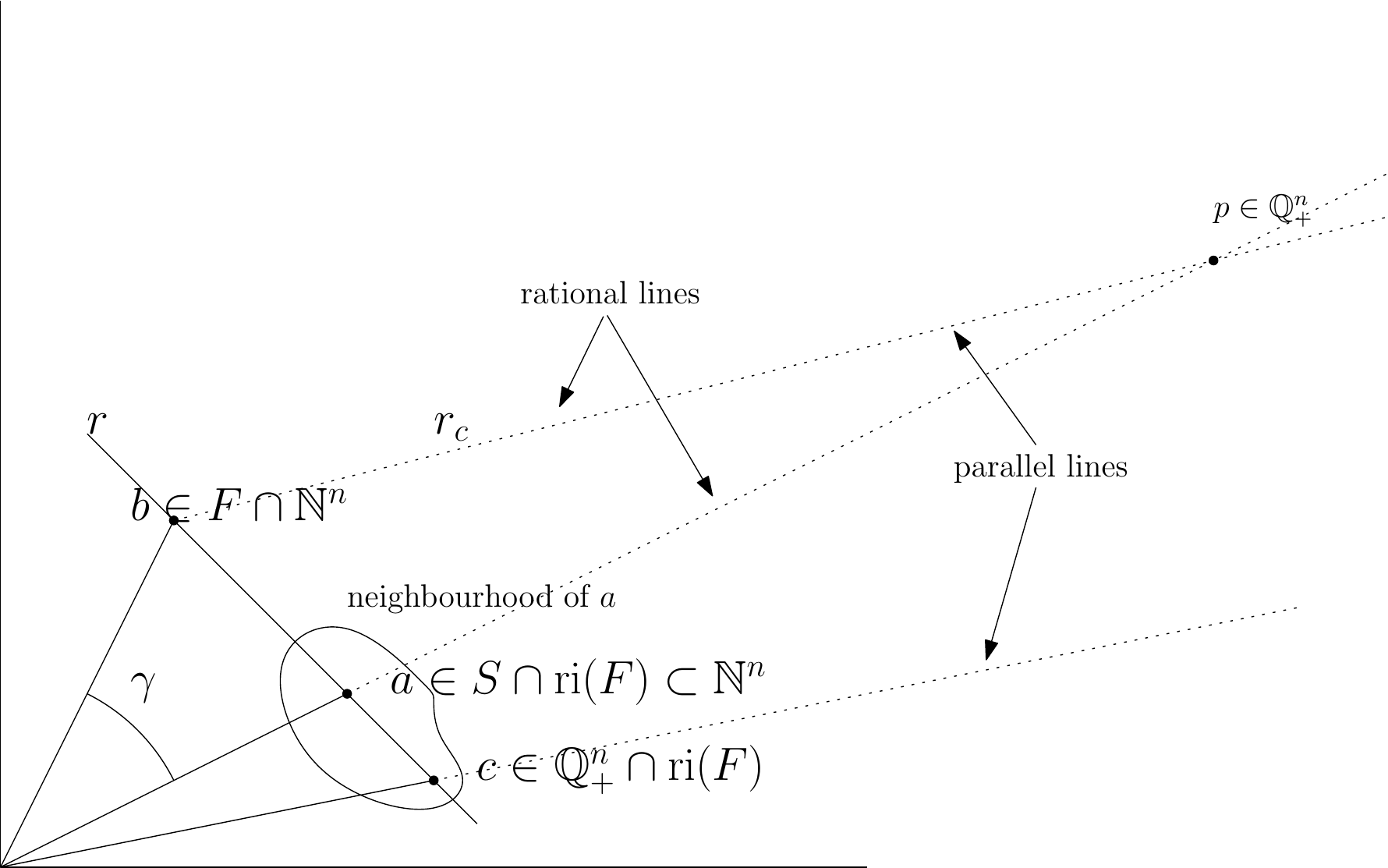}
\caption{}\label{fig1}
\end{center}
\end{figure}

Since
the faces of a polyhedral cone are again polyhedral cones and using
\cite[Theorem 16.4]{schrijver}, we obtain that
$\mathbf C\cap \N^n$ and $F\cap \N^n$ are finitely generated monoids.
Thus, every divisor-closed submonoid of $\mathbf C\cap \N^n$ is a
finitely generated monoid.

Let $\mathbf C\subset \Q_+^n$ be a polyhedral cone. The set
$\mathbf C\cap \N^n$ is an affine semigroup. From the above result and Theorem
\ref{componentes},
we deduce that there exists a bijective correspondence between the lattice of faces of
$\mathbf C$ and the lattice of Archimedean components of $\mathbf C\cap \N^n$.
Given a face $F$, its corresponding Archimedean component is
$(F\setminus \cup\{G\mid G\textrm{ is a face of }\mathbf C \textrm{ such that }
G\subsetneq F\})\cap \N^n$.

\begin{lemma}\label{multiplo}
Let $H\subset \N^n$ be an affine semigroup.
For every $c\in {\rm L}_{\Q_+}(H)$, there exists $m\in \N$ such that
$mc\in H$.
\end{lemma}
\begin{proof}
If $\{g_1,\dots,g_p\}\subset \N^n$ is a system of generators of $H$, then
${\rm L}_{\Q_+}(H)=\{\sum_{i=1}^p\lambda_ig_i\in\Q^n\mid \lambda_i\in\Q_+\}$.
Thus, if $c\in {\rm L}_{\Q_+}(H)$, then there exist $\lambda_1,\dots,\lambda_p\in\Q_+$ such that
$c=\sum_{i=1}^p\lambda_ig_i$. We take $m$ equal to the least common multiple
of the denominators of the elements $\lambda_i$.
\end{proof}

We are now ready to prove the reciprocal of Lemma \ref{lemaS}.

\begin{theorem}\label{carash}
Let $H\subset \N^n$ be an affine semigroup and let $S$ be a submonoid of $H$.
Then,
$S$ is a divisor-closed submonoid of $H$ if and only if
there exists a face $F$ of ${\rm L}_{\Q_+}(H)$ such that $S=F\cap H$.
\end{theorem}
\begin{proof}
Since $H$ is finitely generated, ${\rm L}_{\Q_+}(H)$ is a polyhedral cone
(it is generated as a cone by the generators of $H$).
Let $F$ be the smallest face of ${\rm L}_{\Q_+}(H)$ containing $S$.
We have $S\subseteq F\cap H$.

We prove now that $F\cap H\subseteq S$.
Take $b$ an element of $F\cap H$ and let us see that $b\in S$.
Using Lemma \ref{interior}, there exists
$a\in S\cap {\rm ri}(F)$.
Reasoning as in Proposition \ref{carascono},
there exist $k_1,k_2,k_3\in\N$ and $c\in F$ satisfying
$k_1 a=k_2 b+k_3 c$ and $k_1 a,k_2b, k_3c\in \N^n$.
Since $c\in F\subset {\rm L}_{\Q_+}(H)$,
by Lemma \ref{multiplo} there exists $m\in\N$ such that $mc\in H$. So,
we obtain that $mk_1 a=mk_2 b+mk_3 c\in S$, $mk_1 a\in S$ and
$mk_2 b,mk_3c \in H\cap F$.
Since $S$ is divisor-closed we obtain that $mk_2b\in S$.
Now, using that $mk_2\in\N$ and again that $S$ is divisor-closed we have $b\in S$.
Hence, $F\cap H\subset S$ and therefore
$S=F\cap H$.
\end{proof}

\begin{remark}
By Lemma \ref{dcfg},
every divisor-closed submonoid of an affine semigroup $H$ is finitely generated.
Note that
if $\{g_1,\dots,g_p\}\subset \N^n$ is a system of generators of $H$ and $F$ is a face of ${\rm L}_{\Q_+}(H)$, then
a system of generators of
the divisor-closed submonoid associated to $F$ 
 is the set
$\{g_1,\dots,g_p\}\cap F$.
\end{remark}

The results obtained so far for affine semigroups can be summarized as follows.

\begin{corollary}
Let $H$ be an affine semigroup of $\N^n$.
The lattice of divisor-closed submonoids of $H$, the lattice of Archimedean components of $H$ and
the lattice of faces of the polyhedral cone ${\rm L}_{\Q_+}(H)$ are isomorphic.
\end{corollary}

An affine semigroup $H\subset \N^n$ is simplicial if the cone ${\rm L}_{\Q_+}(H)$ is
generated by $n$ linearly independent generators of $H$. These semigroups are used
to obtain examples of Cohen-Macaulay, Gorenstein and Buchsbaum rings (see \cite{RosalesCM},
\cite{GV_CM_Go}, \cite{OCPS} for further details).
We already know that these semigroups have at most $2^n$ divisor-closed submonoids.
We prove now that they have exactly this number.

\begin{corollary}
Let $H$ be a simplicial submonoid of $\N^n$.
The number of divisor-closed submonoids of $H$ is equal to $2^n$.
\end{corollary}
\begin{proof}
Let $H$ be a simplicial semigroup of $\N^n$, let
 $G=\{g_1,\dots,g_p\}$ be a system of generators of $H$.
Since $H$ is simplicial, the cone ${\rm L}_{\Q_+}(H)$ is generated by $n$ elements
of $G$. Assume that  a system of generators of ${\rm L}_{\Q_+}(H)$ is
$R=\{g_1,\dots, g_n\}\subset G$.
This implies that ${\rm L}_{\Q_+}(H)$ has $n$ bounding hyperplanes
$\mathcal H_i=\langle g_1,\dots,g_{i-1},g_{i+1},\dots, g_n\rangle $ for $1\leq i\leq n$.
Since every $\mathcal H_i$ contains a different facet of ${\rm L}_{\Q_+}(H)$
(just take $\omega$ or $-\omega$ equal to  the a normal vector of $\mathcal H_i$
and consider
${\rm face}_\omega({\rm L}_{\Q_+}(H))$ or
${\rm face}_{-\omega}({\rm L}_{\Q_+}(H))$),
the number of facets of
${\rm L}_{\Q_+}(H)$ is $\binom n {n-1}=n$.

For every facet $F$,
we consider now the $(n-1)$-dimensional subspace containing $F$.
We have a $(n-1)$-cone generated by $n-1$ elements.
So, its $(n-2)$-faces are generated by the subsets with cardinality equal
to  $n-2$ of the system of generators of $F$.
So, the $(n-2)$-faces of the $(n-1)$-faces of ${\rm L}_{\Q_+}(H)$
are generated by the subsets  with cardinality equal to $n-2$ of
$\{g_1,\dots,g_{i-1},g_{i+1},\dots, g_n\}$ for $i\in\{1,\dots,n\}$.
Since every $(n-2)$-face of ${\rm L}_{\Q_+}(H)$ is contained in a $(n-1)$-face of
${\rm L}_{\Q_+}(H)$ and by the transitivity of the relation \emph{is a face of}, we obtain
that the number of $(n-2)$-faces of ${\rm L}_{\Q_+}(H)$ is equal to $\binom n {n-2}$.

Repeating this process we obtain that the total number of faces of ${\rm L}_{\Q_+}(H)$ is
\[ \binom n {n}+\binom n {n-1}+\binom n {n-2}+\dots+\binom n {1}+\binom n {0}=2^n.\]

\end{proof}

\begin{example} Given $H=\langle (1,0),(1,2),(1,3),(1,7) \rangle$,
the cone ${\rm L}_{\Q_+}(H)$ is generated by $\{(1,0),(1,7)\}$.
The faces of ${\rm L}_{\Q_+}(H)$ are $\{(0,0)\}$, the cone $F_1$ generated by $(1,0)$,
the cone $F_2$ generated by $(1,7)$ and ${\rm L}_{\Q_+}(H)$.
Therefore, $\D(H)=\{\{(0,0)\}, H\cap F_1=\langle (1,0) \rangle,H\cap F_2=\langle (1,7)\rangle,H\}$.
Note that $H$ is simplicial and the number of divisor-closed submonoids obtained is $2^2$.

\end{example}

\begin{example}\label{ex-largo}
Let
$H$ be the affine semigroup generated by \[G=\{ (2,14,2),(5,6,1),(7,4,4),(9,3,5),(5,5,15),(6,9,12),(3,9,7),(10,1,3),(3,6,8)\}.\]
Using {\tt normaliz} (see \cite{normaliz}), we obtain that
the extremal rays of $H$ are $\{(1,1,3),(1,7,1),(5,6,1),(10,1,3)\}$. So, the cone
${\rm L}_{\Q_+}(H)$ has $10$ faces:
\[\begin{multlined}
\mathfrak F=\{\{(0,0,0)\},
\langle (1,1,3)\rangle,\langle (1,7,1)\rangle,\langle (5,6,1)\rangle,
\langle(10,1,3) \rangle,\\
\langle (1,7,1),(5,6,1)  \rangle,\langle(5,6,1),(10,1,3) \rangle,
\langle(10,1,3),(1,1,3) \rangle,\langle (1,1,3),(1,7,1)\rangle,
{\rm L}_{\Q_+}(H)\}\}.
\end{multlined}.\]
The equations of the $2$-faces are
\[\{-x-4y+29z=0,-17x+5y+55z=0,3y-z=0,10x-y-3z=0\}.\]
So, the set of divisor-closed submonoids of $H$ is
\[\begin{multlined}
\{\{(0,0,0\},\langle (5,5,15)\rangle,
\langle (2,14,2)\rangle,
\langle (5,6,1)\rangle,
\langle (10,1,3)\rangle,\\
\langle (2,14,2), (5,6,1) \rangle,
\langle (5,6,1), (10,1,3) \rangle,
\langle (5,5,15), (10,1,3) \rangle,\\
\langle (2,14,2), (5,5,15), (3,9,7), (3,6,8) \rangle,H\}.
\end{multlined}\]

\end{example}

\section{
Divisor-closed submonoids of finitely generated cancellative monoids}

By using the results of above sections, we give a different approach for
computing the lattice of divisor-closed submonoids
of a finitely generated cancellative monoid $\widetilde H$.
For that purpose, we compute the lattice of divisor-closed submonoids of an affine semigroup $H$ associated to $\widetilde{H}$.
We show how the monoid $H$ is obtained.

Let $\widetilde H=\N^n/\sim_M$ be a reduced monoid with $M\leq \Z^p$ and
let
\[A=
\left(\begin{array}{ccc}
a_{11} & \dots & a_{1p}\\
\vdots &  & \vdots\\
a_{r1} & \dots & a_{rp}\\
a_{(r+1)1} & \dots & a_{(r+1)p}\\
\vdots &  & \vdots\\
a_{(r+k)1} & \dots & a_{(r+k)p}
\end{array}
\right)\in \mathcal M_{(r+k)\times p}(\Z),~
X=\left(
\begin{array}{c}
x_1 \\ \vdots \\ x_p
\end{array}
\right),~
\mathcal G=
\Z_{d_1} \times \dots \times \Z_{d_r} \times \Z^k
\]
such that the equations of $M$ are  $(A.X)^T=0\in \mathcal G$.
There exists $b=(b_1,\dots,b_p)\in\N^p$ such that $b_1\neq 0,\dots,b_p\neq 0$, and thus
we can add $b$ to all the homogeneous rows of $A$ obtaining an equivalent system of
equations. Since the elements of the $i$th row of $A$ with $1\leq i\leq r$ are in
$\Z_{d_i}$, we assume that all of them are in
$\{0,\dots,d_i-1\}$. So, the matrix $A$ has all its entries in $\N$.

From the columns of $A$ we obtain the elements
$\widetilde a_{*j}=([a_{1j}]_{d_1},\dots,[a_{rj}]_{d_r},a_{(r+1)j},\dots,a_{(r+k)j})\in \Z_{d_1}\times\dots\times\Z_{d_r}\times \Z^{k}$ and the elements
$a_{*j}=(a_{1j},\dots,a_{rj},a_{(r+1)j},\dots,a_{(r+k)j})\in \N^{r+k}$.
By \cite[Proposition 3.1]{Rosales3}, we have that $\N^n/\sim_M$ is isomorphic to the
submonoid 
of $\Z_{d_1}\times\dots\times\Z_{d_r}\times \Z^{k}$ generated by $\{\widetilde a_{*1},\dots, \widetilde a_{*n}\}$.
Define $ H$ as the submonoid of $\N^{r+k}$ generated by the set $\{ a_{*j}\mid 1\leq j\leq n\}$.

Since $\N^{r+k}$ is free, the map
\[\pi:\N^{r+k}\to \Z_{d_1}\times\dots\times\Z_{d_r}\times \N^{k}\]
is a monoid morphism (see \cite{morfismos}) verifying that
$\pi( a_{*j} )=\widetilde a_{*j}$ for all
$1\leq j\leq r+k$. Therefore $\pi( H)=\widetilde H$ and
\[\pi_{|  H}: H\to \widetilde H\]
is a monoid morphism.

Note that for every subset $J\subset \widetilde H$ we have
$\pi^{-1}_{| H}(J)=\pi^{-1}(J)\cap H$.
We see now  the relationship between the submonoids of $\tilde H$ and $H$.

\begin{lemma}
If $S$ is a submonoid of $\widetilde H$, then $\pi^{-1}(S)\cap H$ is a submonoid of
$ H$.
\end{lemma}
\begin{proof}
Let  $s_1,s_2 \in \pi^{-1}(S)\cap H$, then $s_1,s_2\in H$ and
$\pi(s_1),\pi(s_2)\in S$. Therefore
$\pi(s_1)+\pi(s_2)\in S$. Since $\pi(s_1)+\pi(s_2)=\pi(s_1+s_2)$ and $s_1+s_2\in H$,
we obtain that $s_1+s_2\in \pi^{-1}(S)\cap H$.
\end{proof}

The following proposition characterizes divisor-closed submonoids $\widetilde H$ in terms of the divisor-closed submonoids of the affine semigroup $H$.

\begin{proposition}\label{propproy}
Let $\tilde S$ be a submonoid of $\widetilde H$. Then,
 $\tilde S$ is a divisor-closed submonoid of $\widetilde H$ if and only if
 $\pi^{-1}(\tilde S)\cap H$ is a divisor-closed submonoid of $ H$.
\end{proposition}
\begin{proof}
Assume that $\tilde S$ is a divisor-closed submonoid of $\widetilde H$. If
$a,b\in  H$ and
$a+b\in \pi^{-1}(\tilde S)\cap H$, then $\pi(a+b)=\pi(a)+\pi(b)\in \tilde S$. Thus,
$\pi(a),\pi(b) \in \tilde S$, and therefore $a,b\in \pi^{-1}(\tilde S)\cap H$.

Conversely, assume that $\pi^{-1}(\tilde S)\cap H$ is a divisor-closed submonoid of
$ H$. If $a',b'\in\widetilde H$ and $a'+b'\in\tilde  S$, then there exist $a,b\in H$
such that $\pi(a)=a'$ and $\pi(b)=b'$. We have that $\pi(a)+\pi(b)=\pi(a+b)=a'+b'\in \tilde S$,
and therefore $a+b\in \pi^{-1}(\tilde S)\cap H$. Hence, $a,b\in\pi^{-1}(\tilde S)\cap H$
and thus $a',b'\in \tilde S$.

\end{proof}

\begin{corollary}\label{checkdc}
The set of divisor-closed submonoid of $\widetilde H$  is equal to
\begin{equation}\label{conj1}
\{\pi(S)\mid S \textrm{ is a divisor-closed submonoid of } H
\textrm{ and }(\pi^{-1}\circ\pi)(S)\cap H=S\}.
\end{equation}
\end{corollary}
\begin{proof}
By Proposition \ref{propproy},
since $(\pi^{-1}\circ\pi)(S)\cap H=S$ is a divisor-closed submonoid of
$ H$, the monoid $\pi(S)$ is a divisor-closed submonoid of $\widetilde H$.
Thus, every element of (\ref{conj1}) is a divisor-closed submonoid of $\widetilde H$.

Let now $\widetilde S$ be a divisor-closed submonoid of $\widetilde H$, and denote the submonoid $\pi^{-1}(\widetilde S)\cap H$ by $S$.
By Proposition \ref{propproy}, $S$ is
a divisor-closed submonoid of $H$ and, 
clearly, 
$(\pi^{-1}\circ\pi)(S)\cap H= S$ and $\pi (S)=\tilde S$.

\end{proof}

We now illlustrate  the above results with two example where we compute the set divisor-closed submonoids.

\begin{example}\label{exproyeccion}
Let $\widetilde H=\N^4/\sim_M$ with 
$M=\langle (-5,-7,5,7),(12,1,-1,-12),(-5,0,0,5) \rangle $. Using the commands 
\begin{lstlisting}
genM=Matrix([[-5,-7,5,7],[12,1,-1,-12],[-5,0,0,5]])
generatorsToEquations(genM,[x1,x2,x3,x4])	
\end{lstlisting}
from the package \cite{ecuacionesZpython}, we obtain that the defining equations of $M$ are 
\[
\left\{ \begin{array}{rcl}
-2x_1-55x_2-79x_3 & \equiv & 0 ~{\rm mod}~ 10,\\
x_1+x_2+x_3+x_4 & = & 0. 
\end{array}\right.
\]
Observe that since the last equation has all its coefficients greater than zero, the monoid $\widetilde H$ has no units. Furthermore, the first equation is equivalent to $8x_1+5x_2+x_3\equiv 0 ~ {\rm mod}~ 10$, thus, 
\[
\left\{ \begin{array}{rcl}
8x_1+5x_2+x_3 & \equiv & 0~{\rm mod}~ 10,\\
x_1+x_2+x_3+x_4 & = & 0
\end{array}\right.
\] is also a set of defining equations of $M$.
From the comments at the beginning of this section, we obtain that $\widetilde H \cong \langle ([8],1),([5],1),([1],1),([0],1)\rangle \leq \mathbb Z_{10}\times \mathbb Z$ and $H$ is the affine submonoid of $\N^2$ generated by $\{ (8,1),(5,1),(1,1),(0,1)\}$.
Using Theorem \ref{carash}, the set of divisor-closed submonoids of $H$ are $\{\{0\}, S_1=\langle(0,1)\rangle,S_2=\langle (8,1)\rangle,H \}$.
Clearly $\pi(\{0\})=\{0\}$ and $\pi(H)=\widetilde H$ are divisor-closed submonoids of $\widetilde H$, it only remains to check if $\pi(S_1)$ and $\pi(S_2)$ are divisor-closed.
By Corollary \ref{checkdc}, $\pi(S_1)$ is divisor-closed if and only if $(\pi^{-1}\circ\pi)(S_1)\cap H=S_1$.
The monoid $S_1$ is equal to  $\{(0,n)\mid n\in\N\}$, but $(\pi^{-1}\circ\pi)((0,2))\cap H=\{(0,2),(10,2)\}$.  Therefore, $(\pi^{-1}\circ\pi)(S_1)\cap H\neq S_1$ and $S_1$ is not divisor-closed.
For $S_2$, we have that $S_2=\{(8n,n)\mid n\in \N\}$ and 
$(\pi^{-1}\circ\pi)((16,2))\cap H=\{(6,2),(16,2)\}$. Thus, $(\pi^{-1}\circ\pi)(S_2)\cap H\neq S_2$, and $S_2$ is not divisor-closed.
So,  $\D(\widetilde H)=\{\{0\},\widetilde H\}$.
\end{example}

\begin{example}\label{exx}
Another example is given by the monoid $\widetilde H=\N^4/\sim_M$ with $M$ the group 
$\langle (-4,-2,4,4),(5,2,-5,-4),(2,2,-2,-4)\rangle$.
Using the commands 
\begin{lstlisting}
genM=Matrix([[-4,-2,4,4],[5,2,-5,-4],[2,2,-2,-4]])
generatorsToEquations(genM,[x1,x2,x3,x4])
\end{lstlisting}
we obtain that the defining equations of $M$ are 
\[
\left\{
\begin{array}{rcl}
x_2 & \equiv & 0~{\rm mod}~2,\\
x_1 + x_3 & = & 0,\\
2x_2+x_4 & = & 0.
\end{array}
\right.
\]
So, $H$ is isomorphic to the affine semigroup 
$\langle (1,0,2),(0,1,0),(0,0,1)\rangle$ 
and $\widetilde H$ is isomorphic to the submonoid of $\Z_2\times \N^2$ generated by $\{ ([1],0,2),([0],1,0),([0],0,1)\}$.
The set of divisor-closed submonoids of $H$ are the intersections of $H$ with the faces of the cone
${\rm L}_{\Q_+}(H)$ (see Figure \ref{figuraEjemplo})
\begin{multline}
{\mathfrak D}(H)=\{\{0\},H,
S_{11}=\langle (1,0,2) \rangle, S_{12}=\langle (0,1,0) \rangle, S_{13}=\langle (0,0,1) \rangle, \\
S_{21}=\langle (1,0,2),(0,1,0) \rangle,
S_{22}=\langle (1,0,2),(0,0,1) \rangle,
S_{23}=\langle (0,1,0),(0,0,1) \rangle
\}.
\end{multline}

\begin{figure}[h]
	\begin{center} 
		\includegraphics[scale=.3]{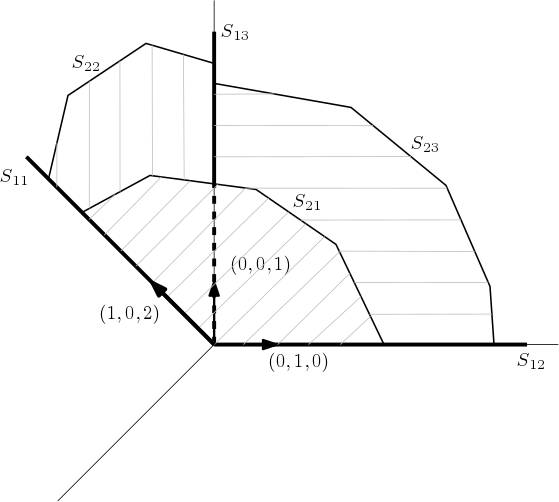}
		\caption{}\label{figuraEjemplo}
	\end{center}
\end{figure}

Note that for every $(x,y,z)\in H$, the set $(\pi^{-1}\circ \pi)((x,y,z))\cap H$ is a subset of $\{ (x,y,z+2x)\mid x,y,z\in \Z \}$.
For the subsemigroups $S_{11}$ and $S_{13}$ we consider the elements $(2,0,4)$ and $(0,0,4)$, respectively. 
We have that $\pi^{-1}(\pi((2,0,4)))\cap H=\pi^{-1}(\pi((0,0,4)))\cap H=\{ (0,0,4),(2,0,4)\}$. Since $(0,0,4)\not \in S_{11}$ and $(2,0,4)\not \in S_{13}$, neither $\pi(S_{11})$ nor $\pi(S_{13})$ are divisor-closed submonoids of $\widetilde H$.
A similar reasoning is used to prove that $\pi(S_{21})$ and $\pi(S_{23})$ are not divisor-closed.
The elements of $S_{12}$ are of the form $(0,n,0)$ and they verify that $(\pi^{-1}\circ \pi)((0,n,0))\cap H=\{(0,n,0)\}$, therefore, $\pi(S_{12})$ is a divisor-closed submonoid of $\widetilde H$.
For every element $(x,y,z)\in S_{22}$, we have now that $(\pi^{-1}\circ \pi)((x,y,z))\cap H\subset S_{22}$. So, the submonoid $\pi(S_{22})$ is a divisor-closed submonoid of $\widetilde H$, and 
$\D(\widetilde H)=\{\{0\},\widetilde H, \pi(S_{12}), \pi(S_{22})\}$.

\end{example}
	


\section{Computing the set of minimal distances}

The results obtained so far are now used to compute the set of minimal distances of a finitely generated cancellative monoid. Before describing this set, some definitions and notation are necessary.

Let $G=\{g_1,\dots,g_p\}$ be a system of generators of $H$, and assume that $H$ is isomorphic to $\N^p/\sim_M$ with $M$ a subgroup of $\Z^p$.  
Denote by ${\sf Z}(h)$ the set $\{(x_1,\dots,x_p)\in\N^p\mid \sum_{i=1}^px_ig_i=h\}$ for every $h\in H$.
We have that for  all $x,y\in\N^p$ and every $h\in H$, if the elements $x$, $y$ belong to ${\sf Z}(h)$ then  $x-y\in M$.
Define
the linear function $|\cdot|:\Q^p\to \Q$ with $|(x_1,\dots,x_p)|=\sum_{i=1}^px_i$.
The set of lengths of $h$ in $H$ is the set 
$$
\mathcal L(h)=\{ |(x_1,\dots,x_p)| : (x_1,\dots,x_p)\in{\sf Z}(h)
\}
.$$
This set is bounded if and only if $M\cap \N^p=\{0\}$, and if so,
there exist  some positive integers $l_1<\dots<l_k$ such that
$\mathcal L(h)=\{l_1,\dots, l_k\}$. 

\begin{definition}
The set
$
\Delta(h)=\{l_i-l_{i-1}: 2\leq i \leq k\}
$
is known as the Delta set of $h$, and 
$
\Delta(H)=\bigcup_{h\in H} \Delta(h)
$
is called the Delta set of $H$.
\end{definition}

It is straightforward to prove that
for every divisor-closed submonoid $S\subseteq H$, we have
$\Delta(S)\subseteq \Delta(H)$.
In \cite{Geroldinger-Qinghai}, we find the following definition.

\begin{definition}
Let $H$ be a Krull monoid, the set of
minimal distances of $H$ is defined as
$$
\Delta^*(H)=\{\min (\Delta(S))| S\subset H \textrm{ is a
divisor-closed submonoid with }\Delta(S)\neq \emptyset\}.
$$
\end{definition}

Clearly, we have $\Delta^*(H)\subset \Delta(H)$, and
$\Delta^*(H)=\emptyset$ if and only if $\Delta(H)=\emptyset$.
We would like to note that since
numerical semigroups do not have any non-trivial divisor-closed submonoids,
for every numerical semigroup $H$ the set 
$\Delta^*(H)$ is equal to $\{\min(\Delta(H))\}$.

Now, we see a method to compute $\min(\Delta(S))$.
\begin{lemma}\label{lema:d}
Let $H=\langle h_1,\dots, h_p\rangle\cong\N^p/\sim_M$ be a monoid
with
$\{ m_1,\dots,m_{r}\}$ a system of
generators of $M$.
Then
\[\min (\Delta(H))=\min\{|m| : |m|>0,~m\in M\}=\gcd(|m_1|,\dots,|m_{r}|).\]
\end{lemma}
\begin{proof}
Since $\Delta(H)\subset \N$, there exists the minimum of $\Delta(H)$.
Let $m\in M$ such that $|m|>0$. There exists $x\in\N^p$ such that $x+m\in\N^p$.
The elements
$x,x+m\in\N^p$ are both in
${\sf Z}(h)$ with $h=[x]_{\sim_M}$, and therefore $|x|,|x+m|\in\mathcal L(h)$.
Since $|x+m|=|x|+|m|$, $\min(\Delta(H))\leq \min(\Delta(h))\leq |x+m|-|x|=|m|$.

Let $h\in H$ and $d\in\Delta (h)$. There exist two different elements
$x,y\in\N^p$
such that
$x,y\in {\sf Z}(h)\subset \N^p$ such that $||x|-|y||=d$. This implies that
$x-y\in M$. If $|x|-|y|>0$, we take $m=x-y$ otherwise $m=y-x$.
Clearly, $|x|=|y+m|$ and $d=|x|-|y|=|m|>0$. Thus
$\min (\Delta(H))\geq\min\{|m|: |m|>0,~m\in M\}$.

Finally,
\[\begin{multlined}
\min\{|m| : |m|>0,~m\in M\}
=\min\{|\sum_{i=1}^r\mu_im_i|:|\sum_{i=1}^r\mu_im_i|>0,~\mu_i\in \Z\}\\
=\min\{\sum_{i=1}^r\mu_i|m_i|\in\N\setminus\{0\}\mid \mu_i\in \Z\}
=\gcd(|m_1|,\dots,|m_{r}|).
\end{multlined}\]
\end{proof}

We now describe an 
algorithm for computing $\Delta^*(H)$.

\begin{algorithm}\label{alg1}
Input: $H\cong \N^p/\sim_M$. Output: $\Delta^*(H)$.
\begin{enumerate}
\item Compute the lattice $\mathfrak D (H)$ of divisor-closed submonoids of $H$.
\item For every $S\in \mathfrak D (H)$, if
$\{[e_{i_1}]_{\sim_M},\dots,[e_{i_t}]_{\sim_M}\}$ is
a system of generators of  $S$,
compute a system of generators $G_S$ of the group obtained from the
intersection of $M$ with
$\{(x_1,\dots,x_p)\in\Z^p\mid x_i=0\textrm{ for all }i\not\in\{i_1,\dots,i_t\} \}$.
\item For every $S\in \mathfrak D (H)$, compute
$|G_S|=\{\sum_{i=1}^p|m_i| : (m_1,\dots,m_p)\in G_S\}$ and $\d_S=\gcd(|G_S|)$.
\item Return $\{d_S\mid S\in \mathfrak D (H) \}$.
\end{enumerate}
\end{algorithm}

We now illustrate the above algorithm with two examples.

\begin{example}
Let $H$ be the affine semigroup generated minimally by
$\{(5, 9, 0), (10, 11, 0), (15, 5, 0), (0, 0, 1), (10, 0, 1)\}$.
The monoid $H$ is isomorphic to $\N^5/\sim_M$ with $M$ the subgroup of $\Z^5$ with
defining equations
\[
\left(
\begin{array}{ccccc}
 5 & 10 & 15 & 0 & 10 \\
 9 & 11 & 5 & 0 & 0 \\
 0 & 0 & 0 & 1 & 1 \\
\end{array}
\right)
.
\left(
\begin{array}{c}
x_1\\\vdots\\x_5
\end{array}
\right)=
\left(
\begin{array}{c}
0\\\vdots\\0
\end{array}
\right).
\]
A system of generators of this group is
$\{(-2,18,-36,-37,37), (-23,202,-403,-414,414)\}$ whose lengths are equal to $-20$ and $-224$.
Therefore $\min (\Delta(H))=\gcd(-20,224)=4$.

The cone ${\rm L}_{\Q_+}(H)$ has four $1$-faces, the rays generated by the elements $(5,9,0)$,
$(15,5,0)$, $(0,0,1)$ and $(10,0,1)$,  and four
$2$-faces, the cones generated by $\{(5,9,0),(15,5,0)\}$,
$\{(15,5,0),(10,0,1)\}$, $\{(10,0,1),(0,0,1)\}$ and $\{(0,0,1),(5,9,0)\}$.
Thus, the divisor-closed submonoids of $H$ are $S_1=\{(0,0,0)\}$,
$S_2=\langle (5,9,0) \rangle$, $S_3=\langle (15,5,0) \rangle$, $S_4=\langle (0,0,1) \rangle$,
$S_5=\langle (10,0,1) \rangle$,
$S_6=\langle (5,9,0),(10,11,0), (15,5,0) \rangle$, $S_7=\langle (15,5,0),(10,0,1)\rangle$,
$S_8=\langle (10,0,1),(0,0,1) \rangle$, $S_9=\langle (0,0,1),(5,9,0) \rangle$, and $H$.
It is easy to prove that for every $S_i$ with $i\neq 6$, the group $G_{S_i}$ is trivial,
and thus $\Delta(S_i)=\emptyset$.
Next, we compute a system of generators of $G_{S_6}$.
This group is defined by the equations
\[
\left(
\begin{array}{ccc}
 5 & 10 & 15 \\
 9 & 11 & 5 \\
 0 & 0 & 0 \\
\end{array}
\right)
.
\left(
\begin{array}{c}
x_1\\x_2\\x_3
\end{array}
\right)=
\left(
\begin{array}{c}
0\\0\\0
\end{array}
\right)
\]
and a system of generators of $G_{S_6}$ is
$\{(23,-22,7)\}$. Thus, $\min (\Delta(S_6))$ is equal to $23-22+7=8$, and
therefore $\Delta^*(H)=\{4,8\}$.
\end{example}

Note that defining equations of the submonoids of $H$ are formed by some of the columns of the defining equations of $H$. For this reason, the minimum of set $\Delta^*(H)$ is always equal to $\min(\Delta(H))$.

\begin{example}
Let $\widetilde H$ be as in Example \ref{exx}.
To compute $\Delta^*(\widetilde H)$, 
from every element of $\D(\widetilde H)=\{\{0\},\widetilde H, 
\pi(S_{12})=\langle ([0],1,0) \rangle,\pi(S_{22})= \langle ([1],0,2), ([0],0,1) \rangle \}$ we compute a system of generators of their associated groups.
Clearly, $\{0\}$ and $\pi(S_{12})$ are free monoids and 
for these semigroup we obtain that $\min(\Delta(\{0\}))=\min(\Delta(\pi(S_{11}))=\emptyset$.
A system of generators of the group of $\widetilde H$ is $\{(-4,-2,4,4),(5,2,-5,-4),(2,2,-2,-4)\}$, since $|(-4,-2,4,4)|=2$ and $|(5,2,-5,-4)|=|(2,2,-2,-4)|=-2$, $\min(\Delta(\widetilde H))=\gcd(2,-2)=2$. 
The group associated to $\pi(S_{22})$ has the following set of defining equations:
\[
\left\{
\begin{array}{rcl}
x_1 &\equiv& 0 ~{\rm mod}~ 2\\
2x_1+x_2 &=&0\\
\end{array}
\right.
.\]
A system of generators of the above group is $\{(2,-4)\}$ and 
so $\min(\Delta(\pi(S_{22})))=2$. 
Hence $\Delta^*(\widetilde H)=\{2\}$.
Observe that the minimum of $\Delta^*(\widetilde H)$ is equal to $\min(\Delta(\widetilde H))$.
\end{example}

\end{document}